\documentclass[10pt,a4paper]{article}
\usepackage{latexsym,amssymb,amsmath,amsthm,amsfonts,amscd}

\newtheorem{thm}{Theorem}[section]
\newtheorem{lem}[thm]{Lemma}
\newtheorem{definition}[thm]{Definition}
\newtheorem{prop}[thm]{Proposition}
\newtheorem{example}[thm]{Example}

\newtheorem{rmk}{Remark}

\renewcommand{\a}{\alpha}
\renewcommand{\b}{\beta}
\newcommand{\ba}{\bar\alpha}
\newcommand{\bb}{\bar\beta}
\newcommand{\ab}{(\alpha,\beta)}
\newcommand{\pab}{\alpha\phi\left(\frac{\beta}{\alpha}\right)}

\newcommand{\gab}{\alpha\phi\left(b^2,\frac{\beta}{\alpha}\right)}
\newcommand{\p}{\phi}
\newcommand{\pt}{\phi_2}
\newcommand{\po}{\phi_1}
\newcommand{\ptt}{\phi_{22}}
\newcommand{\pye}{\phi_{12}}
\newcommand{\RR}{\mathbb R}
\newcommand{\pp}[2]{\frac{\partial{#1}}{\partial{#2}}}

\begin{document}
\title{On a new class of Finsler metrics}

\author{Changtao Yu and Hongmei zhu}
\date{}
\maketitle

\begin{abstract}
In this paper, the geometric meaning of $\ab$-norms is made clear.
On this basis, a new class of Finsler metrics called general
$\ab$-metrics are introduced, which are defined by a Riemannian
metric and a 1-form. These metrics not only generalize $\ab$-metrics
naturally, but also include some metrics structured by R. Bryant.
The spray coefficients formula of some kinds of general
$\ab$-metrics is given and the projective flatness is also
discussed.
\end{abstract}


\section{Introduction}

$\ab$-metrics form a special class of Finsler metrics partly because
they are ``computable''\cite{bacs-cxy-szm-curv}. The researches on
$\ab$-metrics enrich Finsler geometry and the approaches offer
references for further study.

Randers metrics arising from physical applications\cite{rand-onan}
are the simplest $\ab$-metrics. They are expressed in the form
$F=\a+\b$, where $\a=\sqrt{a_{ij}(x)y^iy^j}$ is a Riemannian metric
and $\b=b_i(x)y^i$ is a 1-form with $\|\b\|_\a<1$. The following
Randers metric
\begin{eqnarray}
F=\frac{\sqrt{(1-|x|^2)|y|^2+\langle
x,y\rangle^2}}{1-|x|^2}+\frac{\langle x,y\rangle}{1-|x|^2}
\end{eqnarray}
is called {\em Funk metric}\cite{funk-uber}. It is a projectively
flat Finsler metric on $\mathbb B^n(1)$ with flag curvature
$K=-\frac{1}{4}$. Recall that a Finsler metric $F$ on an open domain
$\mathcal U\subset\RR^n$ is said to be {\em projectively flat}, if
all the geodesics of $F$ are straight lines\cite{css-szm-riem}.

Another important example of $\ab$-metric was given by L.
Berwald\cite{berw-uber},
\begin{eqnarray}
F=\frac{(\sqrt{(1-|x|^2)|y|^2+\langle x,y\rangle^2}+\langle
x,y\rangle)^2}{(1-|x|^2)^2\sqrt{(1-|x|^2)|y|^2+\langle
x,y\rangle^2}}.
\end{eqnarray}
It is of a special kind of $\ab$-metrics in the form
$F=\frac{(\a+\b)^2}{\a}$ with $\|\b\|_\a<1$. Berwald's metric is
also a projectively flat Finsler metric on $\mathbb B^n(1)$ with
flag curvature $K=0$.

The concept of $\ab$-metrics was firstly proposed by M. Matsumoto in
1972 as a direct generalization of Randers metrics\cite{mats-oncr}.
But some basic concepts of $\ab$-metrics were omitted. In section 2,
we make clear the geometric property about the indicatrixes of
$\ab$-metrics. Roughly speaking, a Minkowski norm $F$ is an
$\ab$-norm if and only if the indicatrix of $F$ is a rotation
hypersurface with the rotation axis passing the origin.

The aim of this paper is to study a new class of Finsler metrics
given by
\begin{eqnarray}\label{gab}
F=\gab,
\end{eqnarray}
where $\p=\p(b^2,s)$ is a $C^\infty$ positive function and
$b^2:=\|\b\|^2_\a$. This kind of Finsler metrics generalize
$\ab$-metrics in a natural way. They are a special class of general
$\ab$-metrics which are defined in section 3. But the most important
reason that we are interested in them is that they include some
Finsler metrics constructed by R. Bryant.

Bryant's metrics\cite{brya-some,brya-proj,brya-fins} are rectilinear
Finsler metrics on $S^n$ with flag curvature $K=1$ and given in the
following form with $X\in S^n,Y\in T_XS^n$,
\begin{eqnarray}\label{bryant}
F(X,Y)=\Re\left\{\frac{\sqrt{Q(X,X)Q(Y,Y)-Q(X,Y)^2}}{Q(X,X)}-i\frac{Q(
X,Y)}{Q(X,X)}\right\},
\end{eqnarray}
where
$$Q(X,Y)=x_0y_0+e^{ip_1}x_1y_1+e^{ip_2}x_2y_2+\cdots+e^{ip_n}x_ny_n$$
are complex quadratic forms on $\RR^{n+1}$ for $n\geq2$ with the
parameters satisfying
$$0\leq p_1\leq p_2\leq\cdots\leq p_n<\pi.$$
Note that the branch of the complex square root being used is the
one satisfying $\sqrt{1}=1$ and having the negative real axis as its
branch locus~(cf. \cite{brya-proj}).

The following result is related to Bryant's metrics, where the
constant $r_\mu$ is given by $r_\mu=\frac{1}{\sqrt{-\mu}}$ if
$\mu<0$ and $r_\mu=+\infty$ if $\mu\geq0$.

\begin{thm}\label{thm-3}
The following general $\ab$-metrics are projectively flat on
$\mathbb B^n(r_\mu)$ with $n\geq2$,
\begin{eqnarray}\label{eqn:bry}
F=\Re\frac{\sqrt{(e^{ip}+b^2)\a^2-\b^2}-i\b}{e^{ip}+b^2}\quad(-\frac{\pi}{2}\leq
p\leq\frac{\pi}{2}),
\end{eqnarray}
where $\a$ and $\b$ are given by
\begin{eqnarray}
\a&=&\frac{\sqrt{(1+\mu|x|^2)|y|^2-\mu\langle
x,y\rangle^2}}{1+\mu|x|^2},\label{a}\\
\b&=&\frac{\lambda\langle x,y\rangle+(1+\mu|x|^2)\langle
a,y\rangle-\mu\langle a,x\rangle\langle
x,y\rangle}{(1+\mu|x|^2)^\frac{3}{2}},\label{b}
\end{eqnarray}
in which $\mu$ is the sectional curvature of $\a$, $\lambda$ is a
constant and $a\in\RR^n$ is a constant vector.
\end{thm}

\begin{rmk}
When $\mu=0,\lambda=1,a=0$, the general $\ab$-metrics
(\ref{eqn:bry}) belong to Bryant's metrics in some appropriate
coordinate. One can see section 4 for details. At the same time, we
will point out that the previous metrics (\ref{bryant}) are not
always regular on the whole sphere. Recall that a Finsler metric is
said to be {\em regular}, if its fundamental tensor is positive
definite everywhere.
\end{rmk}

Moreover, we provide a sufficient condition for the general
$\ab$-metrics (\ref{gab}) to be projectively flat. In this paper, a
1-form is called {\em conformal} with respect to a Riemannian metric
if its dual vector field with respect to the Riemannian metric is
conformal.

\begin{thm}\label{thm-2}
Let $F=\gab$ be a general $\ab$-metric on a manifold $M$ with
dimension $n\geq2$. Then $F$ is locally projectively flat if the
following conditions hold:
\begin{enumerate}
\item
The function $\phi(b^2,s)$ satisfies the following partial
differential equation
\begin{eqnarray}\label{8}
\phi_{22}=2(\phi_{1}-s\phi_{12}).
\end{eqnarray}
\item
$\a$ is locally projectively flat, $\b$ is closed and conformal with
respect to $\a$.
\end{enumerate}
\end{thm}
\begin{rmk}\label{r-3}
Note that $\po$ means the derivation of $\p$ with respect to the
first variable $b^2$. On the other hand, a Riemannian metric $\a$ is
locally projectively flat if and only if it is of constant sectional
curvature by Beltrami's theorem\cite{css-szm-riem}.
\end{rmk}

The projective flatness is connected with the Hilbert's Fourth
Problem. Recently, Z. Shen has characterized all the projectively
flat $\ab$-metrics for dimension $n\geq3$\cite{szm-onpr}. The first
author rewrote the $\ab$-metric $F=\frac{(\a+\b)^2}{\a}$ as
$F=\frac{(\sqrt{1+\bar b^2}\ba+\bb)^2}{\ba}$ in his doctoral
dissertation, where $\ba=(1-b^2)\a,\bb=\sqrt{1-b^2}\b$, and proved
that this kind of Finsler metrics are locally projectively flat if
and only if $\ba$ is locally projectively flat while $\bb$ is closed
and conformal with respect to $\ba$.

Moreover, the first author has classified all the locally
projectively flat $\ab$-metrics for dimension $n\geq3$ in his
doctoral dissertation. The results show that the projective flatness
of an $\ab$-metric always arises from that of some Riemannian metric
by doing some special deformations. Therefore, we claim that the
conditions in Theorem \ref{thm-2} are, in a sense, also a necessary
condition for a non-Randers general $\ab$-metric $F=\gab$ to be
locally projectively flat for $n\geq3$.

To be specific, if $F$ is a non-Randers locally projectively flat
general $\ab$-metric, then $F$ can be represented as $F=\gab$ such
that $\p(b^2,s),\a$ and $\b$ satisfy the conditions in Theorem
\ref{thm-2}. For instance, suppose that $F=\frac{(\a+\b)^2}{\a}$ is
a locally projectively flat $\ab$-metric. In this case, the
corresponding function $\p(s)=(1+s)^2$ does not satisfy Eq.
(\ref{8}). Also $\a$ is not locally projectively flat and $\b$ is
not conformal with respect to $\a$ in general~\cite{szm-onpr}. But
if we rewrite $F$ as $F=\frac{(\sqrt{1+\bar b^2}\ba+\bb)^2}{\ba}$,
then the function $\p(\bar b^2,\bar s)=(\sqrt{1+\bar b^2}+\bar s)^2$
satisfies Eq. (\ref{8}) now. Although $F=\frac{(\a+\b)^2}{\a}$ is
simple in this form, the properties of $\a$ and $\b$ are not so
simple. This phenomenon is similar to that of Randers metrics of
constant flag curvature~\cite{db-cr-szm-zerm}.

\section{The geometric meaning of $\ab$-norms}

Let $V$ be an $n$-dimensional vector space. By definition, an
$\ab$-norm on $V$ is a Minkowski norm expressed in the following
form,
$$F=\a\phi(s),\quad s=\frac{\b}{\a},$$
where $\a=\sqrt{a_{ij}y^iy^j}$ is an Euclidean norm and
$\b=b_iy^i\in V^*$ is a linear functional on $V$. The function
$\phi=\phi(s)$ is a $C^\infty$ positive function on some open
interval $(-b_o,b_o)$ satisfying
$$\phi(s)-s\phi'(s)+(b^2-s^2)\phi''(s)>0,\qquad\forall|s|\leq b<b_o,$$
where $b=:\|\b\|_\a$\cite{css-szm-riem}.

Let $\{e_1,e_2,\cdots,e_n\}$ be an orthonormal basis of $\a$. Then
$$\a(y)=\sqrt{(y^1)^2+(y^2)^2+\cdots+(y^n)^2},\qquad y=y^ie_i\in V\cong\RR^n.$$
It is obvious that the orthogonal group $O(n)$ acting on $V$
preserves $\a$. Conversely, a Minkowski norm on $V$ preserved under
the action of $O(n)$ must be Euclidean. In other words, Euclidean
norms are the most symmetric Minkowski norms.

By considering the symmetry of $\ab$-norms, Theorem \ref{thm-1}
shows that the symmetry of $\ab$-norms is just next to that of
Euclidean norms. Firstly, we give a description of the symmetry of a
Minkowski norm.

\begin{definition}
Let $F$ be a Minkowski norm on an $n$-dimensional vector space $V$
and $G$ be a subgroup of $GL(n,\RR)$. Then $F$ is called {\em
$G$-invariant} if the following condition holds for some affine
coordinate $(y^1,y^2,\cdots,y^n)$ of $V$,
\begin{eqnarray}\label{eqn:g}
F(y^1,y^2,\cdots,y^n)=F((y^1,y^2,\cdots,y^n)g),\qquad\forall y\in
V,\forall g\in G.
\end{eqnarray}
\end{definition}

The symmetry of Minkowski norms should be paid more attentions since
it restricts the global symmetry of Finsler manifolds.

\begin{thm}\label{thm-1}
Let $F$ be a Minkowski norm on a vector space $V$ of dimension
$n\geq2$. Then $F$ is an $\ab$-norm if and only if $F$ is
$G$-invariant, where
$$
G=\left\{g \in GL(n,R)~|~g=\left(%
\begin{array}{cc}
  A & 0 \\
  0 & 1 \\
\end{array}%
\right),~A\in O(n-1) \right\}.
$$
\end{thm}
\begin{rmk}
The above theorem is trivial when $n=1$ because every Finsler curve
is of Randers type by the navigation problem.
\end{rmk}
\begin{proof}
Let $F=\pab$ be an $\ab$-norm. Take an orthonormal basis
$\{e_1,e_2,\cdots,e_n\}$ with respect to $\a$, such that
$\ker\b=\mathrm{span}\{e_1,e_2,\cdots,e_{n-1}\}.$ Then
$$F(y)=\sqrt{(y^1)^2+(y^2)^2+\cdots+(y^n)^2}\phi\left(\frac{by^n}{\sqrt{(y^1)^2+(y^2)^2+\cdots+(y^n)^2}}\right),$$
where $y=y^ie_i$ and $b=\|\b\|_\a$. Obviously, $F$ is $G$-invariant.

Conversely, assume that (\ref{eqn:g}) holds for the affine
coordinate $(y^1,y^2,\cdots,y^n)$.\\
Case 1. $n\geq3$.

By restricting $F$ on the linear subspace given by $y^n=0$, one can
obtain an $O(n-1)$-invariant Minkowski norm, which must be Euclidean
by the previous discussions. So we can choose a positive number $a$,
such that the Euclidean norm
$\a=a\sqrt{(y^1)^2+(y^2)^2+\cdots+(y^n)^2}$ on $V$ satisfies
$\a|_{y^n=0}=F|_{y^n=0}$.

For $y\neq0$, define
\begin{eqnarray}\label{eqn:tp}
\tilde\phi(y^1,y^2,\cdots,y^n)=\frac{F(y^1,y^2,\cdots,y^n)}{\a(y^1,y^2,\cdots,y^n)},
\end{eqnarray}
then $\tilde\phi$ is $G$-invariant, i.e.
$$\tilde\p(y^1,y^2,\cdots,y^n)=\tilde\p((y^1,y^2,\cdots,y^n)g),\qquad\forall y\neq0,\forall g\in G.$$
In particular,
$$\tilde\phi(\cos ty^1+\sin ty^2,-\sin ty^1+\cos ty^2,y^3,\cdots,y^n)=\tilde\phi(y^1,y^2,\cdots,y^n).$$
Differentiating the above equality with respect to $t$ and setting
$t=0$, one obtains $\pp{\tilde\phi}{y^1} y^{2} -
\pp{\tilde\phi}{y^2}y^1=0$. The same argument yields
\begin{eqnarray}\label{eqn:sym}
\pp{\tilde\phi}{y^i} y^{j} - \pp{\tilde\phi}{y^j}y^i=0,\qquad 1\leq
i< j\leq n-1.
\end{eqnarray}
Moreover, since $F$ and $\a$ are both positively homogeneous with
degree one, $\tilde\phi$ is positively homogeneous with degree zero,
i.e., $\tilde\phi(\lambda y)=\tilde\phi(y),\forall\lambda>0.$
Differentiating this equality with respect to $\lambda$ and setting
$\lambda=1$, one obtains
\begin{eqnarray}\label{eqn:zero}
\pp{\tilde\phi}{y^i}y^i=0.
\end{eqnarray}
Taking the spherical coordinate transformation
\begin{eqnarray*}
\left\{\begin{array}{ll}
y^1=r\cos\theta^1\cos\theta^2\cdots\cos\theta^{n-2}\cos\theta^{n-1},\\
y^2=r\cos\theta^1\cos\theta^2\cdots\cos\theta^{n-2}\sin\theta^{n-1},\\
\ \ \ \ \ \ \ \cdots\\
y^{n-1}=r\cos\theta^1\sin\theta^2,\\
y^n=r\sin\theta^1,
\end{array}\right.
\end{eqnarray*}
where
$r>0,-\frac{\pi}{2}\leq\theta^\gamma\leq\frac{\pi}{2}(\gamma=1,\cdots,n-2),0\leq\theta^{n-1}<2\pi$,
and using (\ref{eqn:sym}) (\ref{eqn:zero}), we have
\begin{eqnarray*}
\pp{\tilde\phi}{r}&=&\pp{\tilde\phi}{y^i}\pp{y^i}{r}=\pp{\tilde\phi}{y^i}\frac{y^i}{r}=0,\\
\pp{\tilde\phi}{\theta^\gamma}&=&-\pp{\tilde\phi}{y^1}y^{n-\gamma+1}
\cos\theta^{\gamma+1}\cdots\cos\theta^{n-2}\cos\theta^{n-1}\\
&&-\pp{\tilde\phi}{y^2}y^{n-\gamma+1}\cos\theta^{\gamma+1}\cdots\cos\theta^{n-2}\sin\theta^{n-1}-\cdots\\
&&-\pp{\tilde\phi}{y^{n-\gamma}}y^{n-\gamma+1}\sin\theta^{\gamma+1}
+\pp{\tilde\phi}{y^{n-\gamma+1}}r\cos\theta^1\cdots\cos\theta^\gamma\\
&=&-\pp{\tilde\phi}{y^{n-\gamma+1}}y^1\cos\theta^{\gamma+1}\cdots\cos\theta^{n-2}\cos\theta^{n-1}\\
&&-\pp{\tilde\phi}{y^{n-\gamma+1}}y^2\cos\theta^{\gamma+1}\cdots\cos\theta^{n-2}\sin\theta^{n-1}-\cdots\\
&&-\pp{\tilde\phi}{y^{n-\gamma+1}}y^{n-\gamma}\sin\theta^{\gamma+1}
+\pp{\tilde\phi}{y^{n-\gamma+1}}r\cos\theta^1\cdots\cos\theta^\gamma\\
&=&0,\qquad\gamma=2,\cdots,n-2,\\
\pp{\tilde\phi}{\theta^{n-1}}&=&-\pp{\tilde\phi}{y^1}y^2+\pp{\tilde\phi}{y^2}y^1=0.
\end{eqnarray*}
So $\tilde\phi=\tilde\phi(\theta^1)=\phi\left(\frac{y^n}{\a}\right)$
where the function $\p(s)=\tilde\p(\arcsin as)$, which means
$F=\a\phi\left(\frac{y^n}{\a}\right)$ is an
$\ab$-norm.\\
Case 2. $n=2$.

In this case, (\ref{eqn:g}) is equivalent to
$F(y^1,y^2)=F(-y^1,y^2),\forall y\in V.$ This equation implies that
the indicatrix of $F$ is reflection symmetric with respect to
$y^2$-axis. It is easy to see that it means that the function
defined by (\ref{eqn:tp}) has the form
$\tilde\phi=\phi\left(\frac{y^2}{\a}\right)$ for some function
$\phi$.
\end{proof}

\begin{rmk}
(\ref{eqn:tp}) shows that the function $\phi(s)$ contains the
informations about the shape of the indicatrix.
\end{rmk}

By Zermelo's viewpoint~\cite{db-cr-szm-zerm}, we can obtain new
Minkowski norms by shifting the indicatrix of an $\ab$-norm. We call
them {\em navigation $\ab$-norms}. The indicatrix of a navigation
$\ab$-norm is still a rotation hypersurface, but the rotation axis
does not pass the origin in general.

There will not be more discussions about this kind of Minkowski
norms in this paper. It shouldn't be omitted if one study the
properties of $\ab$-metrics besides Randers
metrics\cite{mxh-hlb-oncu,zlf-aloc}, although it may be very
complicated in algebraic form.

\section{General $\ab$-metrics}

Suppose that $F$ is a Finsler metric on a manifold $M$ such that
$F(x,y)$ is an $\ab$-norm on $T_xM$ for any $x\in M$. $F$ is not an
$\ab$-metric in general. This is because the shape of the indicatrix
for different point may be different. This observation leads to the
following definition.

\begin{definition}
Let $F$ be a Finsler metric on a manifold $M$. $F$ is called a {\em
general $\ab$-metric}, if $F$ can be expressed as the form
$F=\a\phi\left(x,\frac{\b}{\a}\right)$ for some $C^\infty$ function
$\phi(x,s)$ where $x\in M$, some Riemannian metric $\a$ and some
1-form $\b$. $F$ is called a {\em (special) $\ab$-metric}, if $F$
can be expressed as $F=\pab$ for some $C^\infty$ function $\phi(s)$,
some Riemannian metric $\a$ and some 1-form $\b$.
\end{definition}

The Finsler metrics in the form (\ref{gab}) become the simplest
class of general $\ab$-metrics except for special $\ab$-metrics.
$\phi(b^2,s)$ is a positive $C^\infty$ function with $b^2,s$ as its
variables and $|s|\leq b<b_o$ as its definitional domain for some
$0<b_o\leq+\infty$. We use $b^2$ instead of $b$ as the first
variable, partly because it is convenient for computations. In the
rest part of this paper, we will focus on this special kind of
general $\ab$-metrics. Firstly, we can obtain the basic facts of the
general $\ab$-metrics immediately from the corresponding ones of
$\ab$-metrics given in \cite{css-szm-riem}.

\begin{prop}\label{fand}
For a general $\ab$-metric $F=\gab$, the fundamental tensor is given
by
$$g_{ij}=\rho a_{ij}+\rho_0b_ib_j+\rho_1(b_i\a_{y^j}+b_j\a_{y^i})-s\rho_1\a_{y^i}\a_{y^j},$$
where
$$\rho=\p(\p-s\pt),\quad\rho_0=\p\ptt+\pt\pt,\quad\rho_1=(\p-s\pt)\pt-s\p\ptt.$$
Moreover,
$$\det(g_{ij})=\p^{n+1}(\p-s\pt)^{n-2}\big(\p-s\pt+(b^2-s^2)\ptt\big)\det(a_{ij}),$$
$$g^{ij}=\rho^{-1}\left\{a^{ij}+\eta b^ib^j+\eta_0\a^{-1}(b^iy^j+b^jy^i)+\eta_1\a^{-2}y^iy^j\right\},$$
where $(g^{ij})=(g_{ij})^{-1},(a^{ij})=(a_{ij})^{-1},b^i=a^{ij}b_j$,
$$\eta=-\frac{\ptt}{\big(\p-s\pt+(b^2-s^2)\ptt\big)},
\qquad\eta_0=-\frac{(\p-s\pt)\pt-s\p\ptt}{\p\big(\p-s\pt+(b^2-s^2)\ptt\big)},$$
$$\eta_1=\frac{\big(s\p+(b^2-s^2)\pt\big)\big((\p-s\pt)\pt-s\p\ptt\big)}{\p^2\big(\p-s\pt+(b^2-s^2)\ptt\big)}.$$
\end{prop}
\begin{proof}
Recall that the fundamental tensor of a Finsler metric $F$ is given
by $g_{ij}=\frac{1}{2}[F^2]_{y^iy^j}$. Note that for a general
$\ab$-metric, the variable $b^2$ is independent of $y$, so one can
get the above formulas immediately from the corresponding ones of
$\ab$-metrics given in \cite{css-szm-riem}.
\end{proof}

\begin{prop}\label{ttt}
Let $M$ be an $n$-dimensional manifold. $F=\gab$ is a Finsler metric
on $M$ for any Riemannian metric $\a$ and 1-form $\b$ with
$\|\b\|_\a<b_o$ if and only if $\p=\p(b^2,s)$ is a positive
$C^\infty$ function satisfying
\begin{eqnarray}\label{ppp}
\p-s\pt>0,\quad\p-s\pt+(b^2-s^2)\ptt>0,
\end{eqnarray}
when $n\geq3$ or
$$\p-s\pt+(b^2-s^2)\ptt>0,$$
when $n=2$, where $s$ and $b$ are arbitrary numbers with $|s|\leq
b<b_o$.
\end{prop}
\begin{proof}
The case $n=2$ is similar to $n\geq3$, so it is omitted here.
Suppose that (\ref{ppp}) holds. Consider a family of functions
$\p_t(b^2,s)=1-t+t\p(b^2,s)$. Let
$F_t=\a\p_t\left(b^2,\frac{\b}{\a}\right)$ and
$g^t_{ij}=\frac{1}{2}\left[F_t^2\right]_{y^iy^j}$, then $F_0=\a$ and
$F_1=F$. It is easy to see that for any $0\leq t\leq1$ and $|s|\leq
b<b_o$,
$$\p_t-s(\p_t)_2=1-t+t(\p-s\pt)>0,$$
$$\p_t-s(\p_t)_2+(b^2-s^2)(\p_t)_{22}=1-t+t\big(\p-s\pt+(b^2-s^2)\ptt\big)>0.$$
Thus $\det(g^t_{ij})>0$ for all $0\leq t\leq1$. Since $(g^0_{ij})$
is positive definite, we conclude that $(g^t_{ij})$ is positive
definite for any $t\in[0,1]$. Therefore, $F_t$ is a Finsler metric
for any $t\in[0,1]$.

Conversely, assume that $F=\gab$ is a Finsler metric for any
Riemannian metric $\a$ and 1-form $\b$ with $b<b_o$. Then
$\phi(b^2,s)$ is positive. By Proposition \ref{fand},
$\det(g_{ij})>0$ is equivalent to
$$(\p-s\pt)^{n-2}\big(\p-s\pt+(b^2-s^2)\ptt\big)>0,$$
which implies $\p-s\pt\neq0$ when $n\geq3$. Since $\p(b^2,0)>0$, the
previous inequality implies that the first inequality in (\ref{ppp})
holds. The second one also holds because $\det(g_{ij})>0$.
\end{proof}
\begin{rmk}
Note that the second inequality in (\ref{ppp}) doesn't imply the
first one, even though it does for special $\ab$-metrics(cf.
\cite{css-szm-riem}).
\end{rmk}

Let $b_{i|j}$ denote the coefficients of the covariant derivative of
$\b$ with respect to $\a$. Let
$$r_{ij}=\frac{1}{2}(b_{i|j}+b_{j|i}),~s_{ij}=\frac{1}{2}(b_{i|j}-b_{j|i}),
~r_{00}=r_{ij}y^iy^j,~s^i{}_0=a^{ij}s_{jk}y^k,$$
$$r_i=b^jr_{ji},~s_i=b^js_{ji},~r_0=r_iy^i,~s_0=s_iy^i,~r^i=a^{ij}r_j,~s^i=a^{ij}s_j,~r=b^ir_i.$$
It is easy to see that $\b$ is closed if and only if $s_{ij}=0$.

\begin{prop}\label{prop:G}
For a general $(\alpha,\beta)$-metric $F=\gab$, its spray
coefficients $G^i$ are related to the spray coefficients $G^i_\a$ of
$\a$ by
\begin{eqnarray*}
G^i&=& G^i_\a+\a Q s^i{}_0+\left\{\Theta(-2\a Q s_0+r_{00}+2\a^2
R r)+\a\Omega(r_0+s_0)\right\}\frac{y^i}{\a}\\
&&+\left\{\Psi(-2\a Q s_0+r_{00}+2\a^2 R
r)+\a\Pi(r_0+s_0)\right\}b^i -\a^2 R(r^i+s^i),
\end{eqnarray*}
where
$$Q=\frac{\pt}{\p-s\pt},\quad R=\frac{\po}{\p-s\pt},$$
$$\Theta=\frac{(\p-s\pt)\pt-s\p\ptt}{2\p\big(\p-s\pt+(b^2-s^2)\ptt\big)},
\quad\Psi=\frac{\ptt}{2\big(\p-s\pt+(b^2-s^2)\ptt\big)},$$
$$\Pi=\frac{(\p-s\pt)\pye-s\po\ptt}{(\p-s\pt)\big(\p-s\pt+(b^2-s^2)\ptt\big)},\quad
\Omega=\frac{2\po}{\p}-\frac{s\p+(b^2-s^2)\pt}{\p}\Pi.$$
\end{prop}
\begin{proof}
Recall that the spray coefficients of a Finsler metric $F$ are given
by
$$G^i=\frac{1}{4}g^{il}\left\{\left[F^2\right]_{x^ky^l}y^k-\left[F^2\right]_{x^l}\right\}.$$

For the general $\ab$-metric $F=\gab$, direct computations yield
\begin{eqnarray*}
\left[F^2\right]_{x^k}&=&[\a^2]_{x^k}\p^2+2\a^2\p\po[b^2]_{x^k}+2\a^2\p\pt s_{x^k},\\
\left[F^2\right]_{x^ky^l}&=&[\a^2]_{x^ky^l}\p^2+2[\a^2]_{x^k}\p\pt
s_{y^l}+2[\a^2]_{y^l}\p\po[b^2]_{x^k}\\
&&+2\a^2\po\pt[b^2]_{x^k}s_{y^l}+2\a^2\p\pye[b^2]_{x^k}s_{y^l}+2[\a^2]_{y^l}\p\pt s_{x^k}\\
&&+2\a^2(\pt)^2s_{x^k}s_{y^l}+2\a^2\p\ptt
s_{x^k}s_{y^l}+2\a^2\p\pt s_{x^ky^l}.
\end{eqnarray*}
Set $G^i=G^i_1+G^i_2$, where $G^i_1$ includes $\po$ and $\pye$ but
$G^i_2$ does not, i.e.,
\begin{eqnarray}
G^i_1&=&\frac{1}{2}g^{il}\Big\{[\a^2]_{y^l}\p\po[b^2]_{x^k}y^k+\a^2\po\pt[b^2]_{x^k}y^ks_{y^l}\nonumber\\
&&+\a^2\p\pye[b^2]_{x^k}y^ks_{y^l}-\a^2\p\po[b^2]_{x^l}\Big\}.\label{G1}
\end{eqnarray}

It is easy to see that $G^i_2$ can be obtained immediately by
exchanging $\phi'$ for $\pt$ and $\phi''$ for $\ptt$ in the spray
coefficients of $\ab$-metrics which can be found in
\cite{css-szm-riem}. So
\begin{eqnarray*}
G^i_2=G_\a^i+\a Q s^i{}_0+\Theta\left\{-2\a Qs_0+r_{00}\right\}
\frac{y^i}{\a}+\Psi\left\{-2\a Qs_0+r_{00}\right\}b^i.
\end{eqnarray*}

In order to compute $G^i_1$, we need the following simple facts:
\begin{eqnarray}\label{fact}
[\a^2]_{y^l}=2y_l,\quad[b^2]_{x^l}=2(r_l+s_l),\quad s_{y^l}=\frac{\a
b_l -sy_l}{\a^2},
\end{eqnarray}
where $y_l=a_{lt}y^t$.\\
By (\ref{G1}) and (\ref{fact}), we have
\begin{eqnarray*}
G^i_1=g^{il}\big\{A y_l+B b_l+C(r_l+s_l)\big\}:=\rho^{-1}\big\{D
y^i+E b^i+F(r^i+s^i)\big\},
\end{eqnarray*}
where
\begin{eqnarray*}
&A=(2\p\po-s\po\pt-s\p\pye)(r_0+s_0),&\\
&B=\a(\po\pt+\p\pye)(r_0+s_0),\quad C=-\a^2\p\po,&
\end{eqnarray*}
and by Proposition \ref{fand},
\begin{eqnarray*}
D&=&A+(As+\a^{-1}Bb^2+\a^{-1}Cr)\eta_0+\big\{A+\a^{-1}Bs+\a^{-2}C(r_0+s_0)\big\}\eta_1,\\
E&=&B+(\a As+Bb^2+Cr)\eta+\big\{\a
A+Bs+\a^{-1}C(r_0+s_0)\big\}\eta_0,\\
F&=&C.
\end{eqnarray*}
Plugging $\eta,\eta_0,\eta_1,A,B,C$ into the above equalities yields
\begin{eqnarray*}
D&=&\Bigg\{\left[2(\p-s\pt)+\frac{s\ptt\big(s\p+(b^2-s^2)\pt\big)}{\p-s\pt+(b^2-s^2)\ptt}\right]
\po\\
&&-\frac{(\p-s\pt)\big(s\p+(b^2-s^2)\pt\big)}{\p-s\pt+(b^2-s^2)\ptt}
\pye\Bigg\}(r_0+s_0)\\
&&+\frac{(\p-s\pt)\pt-s\p\ptt}{\p-s\pt+(b^2-s^2)\ptt}\po\a r,\\
E&=&\Bigg\{\frac{\p(\p-s\pt)}{\p-s\pt+(b^2-s^2)\ptt}\pye-\frac{s\p\ptt}{\p-s\pt+(b^2-s^2)\ptt}
\po\Bigg\}\a(r_0+s_0)\\
&&+\frac{\p\ptt}{\p-s\pt+(b^2-s^2)\ptt}\po\a^2r.
\end{eqnarray*}
One can obtain the spray coefficients $G^i$ by the above
equalities.
\end{proof}

\section{Some constructions of projectively flat general
$\ab$-metrics}

Bryant's metrics (\ref{bryant}) contain some general $\ab$-metrics.
In order to see that, let us take $p_1=p_2=\cdots=p_{n-1}=0,p_n=p$.
Then (\ref{bryant}) is given in the following form in some
appropriate coordinate by stereographic projection,
$$F=\Re\frac{\sqrt{(e^{ip}+|x|^2)|y|^2-\langle x,y\rangle^2}-i\langle x,y\rangle}{e^{ip}+|x|^2}.$$
If we take $p_1=p_2=\cdots=p_n=p$, then (\ref{bryant}) is given by
$$F=\Re\frac{\sqrt{(e^{-ip}+|x|^2)|y|^2-\langle x,y\rangle^2}-i\langle x,y\rangle}{e^{-ip}+|x|^2}.$$
So it is natural to consider the general $\ab$-metrics in the form
(\ref{eqn:bry}).

\begin{lem}
$F=\Re\frac{\sqrt{(e^{ip}+b^2)\a^2-\b^2}-i\b}{e^{ip}+b^2}$ is a
Finsler metric if and only if $b<b_o$, where
\begin{eqnarray*}
b_o=\left\{\begin{array}{ll}
+\infty&\qquad\mbox{if}\quad|p|\leq\frac{\pi}{2},\\
\sqrt{\frac{1}{2}\sec(\frac{2\pi}{3}-\frac{|p|}{3})}&\qquad\mbox{if}\quad\frac{\pi}{2}<|p|<\pi.
\end{array}\right.
\end{eqnarray*}
\end{lem}

\begin{proof}
There is no need to be discussed when $p=0$, because in this case
$F=\frac{\sqrt{(1+b^2)\a^2-\b^2}}{1+b^2}$ is just a Riemannian
metric.

Define a complex-valued function $\Phi(b^2,s)$ by
\begin{eqnarray}\label{eqn:Phi}
\Phi(b^2,s)=\frac{\sqrt{e^{ip}+b^2-s^2}-is}{e^{ip}+b^2}=\frac{1}{\sqrt{e^{ip}+b^2-s^2}+is},
\end{eqnarray}
then $\phi(b^2,s)$ is the real part of $\Phi$. Direct computations
yield
\begin{eqnarray}\label{eqn:PP}
\Phi-s\Phi_2=\frac{1}{(e^{ip}+b^2-s^2)^\frac{1}{2}},\label{eqn:P1}\\
\Phi-s\Phi_2+(b^2-s^2)\Phi_{22}=\frac{e^{ip}}{(e^{ip}+b^2-s^2)^\frac{3}{2}}.\label{eqn:P2}
\end{eqnarray}
When $0<p<\pi$, it is easy to see that the argument of
$e^{ip}+b^2-s^2$, denoted by $\theta$, satisfies $0<\theta\leq p$
since $b^2-s^2\geq0$. We conclude $\p$ and $\p-s\pt$ are positive
because the arguments of $\Phi$ and $\Phi-s\Phi_2$ belong to the
interval $(-\frac{\pi}{2},\frac{\pi}{2})$.

On the other hand,
$$\arg\left(\Phi-s\Phi_2+(b^2-s^2)\Phi_{22}\right)=p-\frac{3}{2}\theta,$$
so $\p-s\pt+(b^2-s^2)\ptt$ is positive when $p\leq\frac{\pi}{2}$. In
other words, $b_o=+\infty$ when $0<p\leq\frac{\pi}{2}$.

In the case $p>\frac{\pi}{2}$, $\p-s\pt+(b^2-s^2)\ptt$ is not always
positive because $\theta$ may be very small. Let $b_o$ be the
largest number such that for all $|s|\leq b<b_o$,
$\p-s\pt+(b^2-s^2)\ptt>0$. Then $b_o$ must be the solution, which is
given in the lemma, of the following equation,
$$\arg\frac{e^{ip}}{(e^{ip}+b_o^2)^\frac{3}{2}}=\frac{\pi}{2}.$$

We can finish the proof by the similar argument for the case
$-\pi<p<0$.
\end{proof}
\begin{rmk}
By the above lemma, Bryant's metrics (\ref{bryant}) do not always
define on the whole sphere. This conclusion have been confirmed by
R. Bryant. That is to say, in order to ensure the regularity of
(\ref{bryant}) on the whole sphere, some more conditions on the
parameters $p_i(1\leq i\leq n)$ should be provided.
\end{rmk}

\begin{proof}[Proof of Theorem \ref{thm-2}]
Since $\a$ is locally projectively flat, we can assume that
$G^i_\a=\theta y^i$ in some local coordinate system $(\mathcal
U;x^i)$, where $\theta=\theta_i(x)y^i$ is a 1-form on $\mathcal U$.
On the other hand, $b_{i|j}=c(x)a_{ij}$ for some function $c(x)$
because $\b$ is closed and conformal with respect to $\a$. It is
obvious that
\begin{eqnarray}\label{tu}
r_{00}=c\a^2,r_0=c\b,r=cb^2,r^i=cb^i,s^i{}_0=0,s_0=0,s^i=0.
\end{eqnarray}
Substituting (\ref{tu}) into the spray coefficients in Proposition
\ref{prop:G} yields
\begin{eqnarray*}
G^i&=&\left\{\theta+c\a[\Theta(1+2Rb^2)+s\Omega]\right\}y^i+c\a^2\left\{\Psi(1+2Rb^2)+s\Pi-R\right\}b^i\\
&=&\left\{\theta+c\a\left[\frac{\pt+2s\po}{2\p}
-\frac{\big(\ptt-2(\po-s\pye)\big)\big(s\p+(b^2-s^2)\pt\big)}
{2\p\big(\p-s\pt+(b^2-s^2)\ptt\big)}\right]\right\}y^i\\
&&+c\a^2\left\{\frac{\ptt-2(\po-s\pye)}{2\big(\p-s\pt+(b^2-s^2)\ptt\big)}\right\}b^i.
\end{eqnarray*}
So the spray coefficients are given by
\begin{eqnarray}\label{GGG}
G^i=\left\{\theta+c\a\frac{\pt+2s\po}{2\p}\right\}y^i
\end{eqnarray}
if $\p$ satisfies the first condition of Theorem \ref{thm-2}. Recall
that a Finsler metric is projectively flat if and only if its spray
coefficients are in the form $G^i=Py^i$\cite{css-szm-riem}.
Therefore $F$ is projectively flat on $\mathcal U$.
\end{proof}

\begin{proof}[Proof of Theorem \ref{thm-3}]
The function $\Phi(b^2,s)$ is defined by (\ref{eqn:Phi}).
Differentiating (\ref{eqn:P1}) with respect to $b^2$ yields
$$\Phi_1-s\Phi_{12}=-\frac{1}{2(e^{ip}+b^2-s^2)^\frac{3}{2}}.$$
So by the above equality and (\ref{eqn:P2}), $\Phi$ satisfies the
following equality,
$$\Phi_{22}=2(\Phi_1-s\Phi_{12}).$$
The same relation is true for $\p$ by taking the real parts of the
above equality.

On the other hand, set $\varrho=\sqrt{1+\mu|x|^2}$, then the
Christoffel symbols of (\ref{a}) are given by
$\Gamma^k{}_{ij}=-\varrho^{-2}\mu(x^i\delta^k{}_j+x^j\delta^k{}_i)$,
and
\begin{eqnarray*}
b_i&=&\varrho^{-3}\lambda x^i+\varrho^{-1}a^i-\varrho^{-3}\mu\langle a,x\rangle x^i,\\
\pp{b_i}{x^j}&=&\varrho^{-3}\lambda\delta_{ij}-3\varrho^{-5}\mu\lambda
x^ix^j-\varrho^{-3}\mu a^ix^j\\
&&-\varrho^{-3}\mu\langle
a,x\rangle\delta_{ij}-\varrho^{-3}\mu
a^jx^i+3\varrho^{-5}\mu^2\langle a,x\rangle x^ix^j,\\
b_{i|j}&=&\pp{b_i}{x^j}-b_k\Gamma^k{}_{ij}\\
&=&\varrho^{-3}(\lambda-\mu\langle
a,x\rangle)\delta_{ij}-\varrho^{-5}(\lambda-\mu\langle
a,x\rangle)\mu x^ix^j.
\end{eqnarray*}
The last equality implies $s_{ij}=0$ and
$r_{ij}=\varrho^{-1}(\lambda-\mu\langle a,x\rangle)a_{ij}$. So $\b$
is closed and conformal with respect to $\a$ with conformal factor
$c(x)=\varrho^{-1}(\lambda-\mu\langle a,x\rangle)$.

Moreover, the spray coefficients of $F$ are given by
$$G^i=\left\{-\frac{\mu\langle x,y\rangle}{1+\mu|x|^2}+\frac{(\lambda-\mu\langle
a,x\rangle)}{\sqrt{1+\mu|x|^2}}
\Im\frac{\sqrt{(e^{ip}+b^2)\a^2-\b^2}-i\b}{e^{ip}+b^2}\right\}y^i,$$
which are obtained by the simple equality $\Phi_2+2s\Phi_1=-i\Phi^2$
and (\ref{GGG}).
\end{proof}

\begin{example}
Take $\lambda=1,a=0$ in Theorem \ref{thm-3}, then the following
general $\ab$-metrics are projectively flat for $-\frac{\pi}{2}\leq
p\leq\frac{\pi}{2}$:
$$F=\Re\frac{\sqrt{(e^{ip}+|x|^2+\mu e^{ip}|x|^2)|y|^2-(1+\mu e^{ip})\langle
x,y\rangle^2} -\frac{i\langle
x,y\rangle}{\sqrt{1+\mu|x|^2}}}{e^{ip}+|x|^2+\mu e^{ip}|x|^2}.$$
\end{example}

\begin{example}\label{ex2}
It is easy to verify that the function
$\p(b^2,s)=(\sqrt{1+b^2}+s)^2$ satisfies the first condition of
Theorem \ref{thm-2}. Take $\lambda=1,a=0$, then the following
general $\ab$-metrics are projectively flat:
$$F=\frac{(\sqrt{1+(1+\mu)|x|^2}\sqrt{(1+\mu|x|^2)|y|^2-\mu\langle x,y\rangle^2}+\langle x,y\rangle)^2}
{(1+\mu|x|^2)^2\sqrt{(1+\mu|x|^2)|y|^2-\mu\langle x,y\rangle^2}}.$$
In particular, $F$ is the Berwald's metric when $\mu=-1$.
\end{example}

\section{Some discussions about the PDE}

In this section, we will discuss some interesting properties about
the partial differential equation
\begin{eqnarray}\label{pde}
\ptt=2(\po-s\pye).
\end{eqnarray}
We will always assume $\lambda=1$ and $a=0$ in Theorem \ref{thm-3}
in this section. In this case, $\a$ and $\b$ are given by
$$\a_\mu=\frac{\sqrt{(1+\mu|x|^2)|y|^2-\mu\langle x,y\rangle^2}}{1+\mu|x|^2},
\qquad\b_\mu=\frac{\langle x,y\rangle}{(1+\mu|x|^2)^\frac{3}{2}}.$$
It is easy to verify that
$b^2_\mu:=\|\b_\mu\|^2_{\a_\mu}=\frac{|x|^2}{1+\mu|x|^2}$.

For any solution $\p$ of (\ref{pde}) satisfying Proposition
\ref{ttt}, $F=\a_\mu\p\left(b^2_\mu,\frac{\b_\mu}{\a_\mu}\right)$ is
a projectively flat general $\ab$-metric for any constant $\mu$ by
Theorem \ref{thm-2}. It is easy to see that such a metric can always
be rewrote as the form
\begin{eqnarray}
F=|y|\p_\mu\left(|x|^2,\frac{\langle x,y\rangle}{|y|}\right),
\end{eqnarray}
where the function $\p_\mu$ is given by
\begin{eqnarray}\label{T}
\p_\mu(b^2,s)=\frac{\sqrt{1+\mu(b^2-s^2)}}{1+\mu
b^2}\p\left(\frac{b^2}{1+\mu b^2},\frac{s} {\sqrt{1+\mu
b^2}\sqrt{1+\mu(b^2-s^2)}}\right).
\end{eqnarray}
In particular, $\p_0=\p$.

(\ref{T}) defines a family of transformations $\{\mathcal T_\mu\}$
by $\p_\mu=\mathcal T_\mu(\p)$. Such a family of transformations
become a transformation group of the solution space of (\ref{pde})
by the following proposition.

\begin{prop}\label{group}
For any solution $\p(b^2,s)$ of (\ref{pde}), the following facts
hold:
\begin{enumerate}
\item
$\p_\mu=\mathcal T_\mu(\p)$ is also a solution of (\ref{pde}) for
any constant $\mu$;
\item
$\mathcal T_0(\p)=\p$;
\item
$\mathcal T_\mu\circ\mathcal T_\nu(\p)=\mathcal T_{\mu+\nu}(\p)$.
\end{enumerate}
\end{prop}
\begin{proof}
Denote $\p_\mu$ by $\tilde\p$ and set $\tilde\p=A\p(B,S)$ where
\begin{eqnarray*}
&\displaystyle A(b^2,s)=\frac{\sqrt{1+\mu(b^2-s^2)}}{1+\mu b^2},\quad B(b^2,s)=\frac{b^2}{1+\mu
b^2},&\\
&\displaystyle S(b^2,s)=\frac{s}{\sqrt{1+\mu b^2}\sqrt{1+\mu(b^2-s^2)}}.&
\end{eqnarray*}
Then
\begin{eqnarray*}
\tilde\p_2&=&A_2\p(B,S)+AS_2\p_S(B,S)\\
&=&-\frac{\mu s\p(B,S)}{(1+\mu
b^2)\sqrt{1+\mu(b^2-s^2)}}+\frac{\p_S(B,S)}{\sqrt{1+\mu
b^2}\big(1+\mu(b^2-s^2)\big)},\\
\tilde\p-s\tilde\p_2&=&\frac{1}{\sqrt{1+\mu(b^2-s^2)}}\big(\p(B,S)-S\p_S(B,S)\big).
\end{eqnarray*}
Set $E=\frac{1}{\sqrt{1+\mu(b^2-s^2)}}$, then
\begin{eqnarray}
\big(\tilde\p-s\tilde\p_2\big)_1&=&E\big(\p(B,S)-S\p_S(B,S)\big)_BB_1
+E\big(\p(B,S)-S\p_S(B,S)\big)_SS_1\nonumber\\
&&+E_1\big(\p(B,S)-S\p_S(B,S)\big),\label{18}\\
\big(\tilde\p-s\tilde\p_2\big)_2&=&E\big(\p(B,S)-S\p_S(B,S)\big)_SS_2
+E_2\big(\p(B,S)-S\p_S(B,S)\big).\label{19}
\end{eqnarray}
The fact that $\p$ is a solution of (\ref{pde}) yields
\begin{eqnarray}\label{20}
\big(\p(B,S)-S\p_S(B,S)\big)_S=-2S\big(\p(B,S)-S\p_S(B,S)\big)_B.
\end{eqnarray}
Then by (\ref{18}), (\ref{19}) and (\ref{20}) we have
\begin{eqnarray*}
&&2(\tilde\p_1-s\tilde\p_{12})-\tilde\p_{22}\\
&=&2\big(\tilde\p-s\tilde\p_2\big)_1+s^{-1}\big(\tilde\p-s\tilde\p_2\big)_2\\
&=&(2ES_1+s^{-1}ES_2)\big(\p(B,S)-S\p_S(B,S)\big)_S+2EB_1\big(\p(B,S)-S\p_S(B,S)\big)_B\\
&&+(2E_1+s^{-1}E_2)\big(\p(B,S)-S\p_S(B,S)\big)\\
&=&2E(B_1-2SS_1-s^{-1}SS_2)\big(\p(B,S)-S\p_S(B,S)\big)_B\\
&&+(2E_1+s^{-1}E_2)\big(\p(B,S)-S\p_S(B,S)\big)\\
&=&0.
\end{eqnarray*}
The last equality holds because the items $B_1-2SS_1-s^{-1}SS_2$ and
$2E_1+s^{-1}E_2$ are both equal to 0 by direct computations. So (1)
holds.

(2) holds since $\p_0=\p$.

In order to see that (3) is true, we only need to compute $\mathcal
T_\mu(\p_\nu)$. By (\ref{T}) and the definition of $\mathcal T_\mu$,
\begin{eqnarray*}
\mathcal T_\mu(\p_\nu)&=&\frac{\sqrt{1+\nu(b^2-s^2)}}{1+\nu
b^2}\frac{\sqrt{1+\mu\left(\frac{b^2}{1+\nu b^2}-\frac{s^2}{(1+\nu
b^2)\left(1+\nu(b^2-s^2)\right)}\right)}}{1+\mu\frac{b^2}{1+\nu b^2}}\\
&&\p\left(\frac{\frac{b^2}{1+\nu b^2}}{1+\mu\frac{b^2}{1+\nu
b^2}},\frac{\frac{s}{\sqrt{1+\nu
b^2}{\sqrt{1+\nu(b^2-s^2)}}}}{\sqrt{1+\mu\frac{b^2}{1+\nu
b^2}}\sqrt{1+\mu\left(\frac{b^2}{1+\nu b^2}-\frac{s^2}{(1+\nu
b^2)\left(1+\nu(b^2-s^2)\right)}\right)}}\right)\\
&=&\p_{\mu+\nu}(b^2,s),
\end{eqnarray*}
which means $\mathcal T_\mu\circ\mathcal T_\nu(\p)=\mathcal
T_{\mu+\nu}(\p)$.
\end{proof}



Proposition \ref{group} implies a simple fact. If $\tilde\p$ can be
obtained from some solution $\p$ of (\ref{pde}) by some
transformation $\mathcal T_\mu$, then they will offer the same
projectively flat Finsler metrics by Theorem \ref{thm-2}. For
instance, obviously $\p=1$ is a solution of (\ref{pde}), and
$\mathcal T_\mu(1)=\frac{\sqrt{1+\mu(b^2-s^2)}}{1+\mu b^2}$. In this
case,
$$\p_\mu\left(b^2_\nu,\frac{\b_\nu}{\a_\nu}\right)=\p_{\mu+\nu}\left(b^2_0,\frac{\b_0}{\a_0}\right)=\a_{\mu+\nu}$$
are just the Riemannian metrics of constant sectional curvature.

We still don't know how to solve the equation (\ref{pde})
completely, but the following lemma is helpful to get its solutions.

\begin{lem}\label{C}
For any $C^\infty$ functions $f$ and $g$, the following function is
the solution of (\ref{pde}):
\begin{eqnarray}\label{eqn:jie}
\p(b^2,s)=f(b^2-s^2)+2s\int_0^sf'(b^2-\sigma^2)\mathrm{d}\sigma+g(b^2)s.
\end{eqnarray}
\end{lem}
\begin{proof}
It is easy to verify that the above function satisfies
(\ref{pde}).
\end{proof}

Suppose that $\p$ satisfies (\ref{eqn:jie}). Direct computations
show that
$$\p-s\pt=f(t),\qquad\p-s\pt+(b^2-s^2)\ptt=f(t)+2tf'(t),$$
where $t=b^2-s^2\geq0$. Assume that $f(0)>0$, then the inequalities
$\p>0$ and (\ref{ppp}) always hold for $b$ small enough. So one can
construct infinitely many projectively flat general $\ab$-metrics by
Lemma \ref{C}. Some simple examples are given in the following:

\begin{itemize}
\item$f(t)=\frac{1}{\sqrt{1-t}}$,
$$\p(b^2,s)=\frac{\sqrt{1-b^2+s^2}}{1-b^2}+g(b^2)s.$$
In this case, $F$ is of Randers type. In particular, it is the
navigation representation of Randers metrics when
$g(b^2)=-\frac{1}{1-b^2}$(cf. \cite{css-szm-riem}).
\item$f(t)=1+t$,
$$\p(b^2,s)=1+b^2+s^2+g(b^2)s.$$
In particular, it is given by example \ref{ex2} when
$g(b^2)=2\sqrt{1+b^2}$.
\item$f(t)=\sqrt{1-t},$
$$\p(b^2,s)=\sqrt{1-b^2+s^2}-s\ln(\sqrt{1-b^2+s^2}+s)+s\ln\sqrt{1-b^2}+g(b^2)s.$$
\item$f(t)=\sqrt{1+t}$,
$$\p(b^2,s)=\sqrt{1+b^2-s^2}+s\arcsin\frac{s}{\sqrt{1+b^2}}+g(b^2)s.$$
\item$f(t)=\ln(2+t)$,
$$\p(b^2,s)=\ln(2+b^2-s^2)+\frac{s}{\sqrt{2+b^2}}\ln\frac{\sqrt{2+b^2}+s}{\sqrt{2+b^2}-s}+g(b^2)s.$$
\item$f(t)=\ln(2-t)$,
$$\p(b^2,s)=\ln(2-b^2+s^2)-\frac{2s}{\sqrt{2-b^2}}\arctan\frac{s}{\sqrt{2-b^2}}+g(b^2)s.$$
\item$f(t)=1+\arctan t$,
\begin{eqnarray*}
\p(b^2,s)&=&1+\arctan(b^2-s^2)+\frac{s}{\sqrt{1+b^4}\sqrt{2\sqrt{1+b^4}-2b^2}}\\
&&\cdot\Bigg(\frac{1}{2}\left(\sqrt{1+b^4}-b^2\right)\ln\frac{\sqrt{1+b^4}+\sqrt{2\sqrt{1+b^4}+2b^2}s+s^2}
{\sqrt{1+b^4}-\sqrt{2\sqrt{1+b^4}+2b^2}s+s^2}\\
&&+\arctan\left(\sqrt{2\sqrt{1+b^4}+2b^2}s+\sqrt{1+b^4}+b^2\right)\\
&&+\arctan\left(\sqrt{2\sqrt{1+b^4}+2b^2}s-\sqrt{1+b^4}-b^2\right)\Bigg)+g(b^2)s.\\
\end{eqnarray*}
\end{itemize}

Obviously, the general $\ab$-metrics include all the $\ab$-metrics.
But it seems a little difficult to determine whether a general
$\ab$-metric is an $\ab$-metric or not. If $\p=\p(s)$ is independent
of $b^2$, then there is no doubt that $F=\pab$ is an $\ab$-metric.
But if $\p=\p(b^2,s)$, we can't conclude that $F=\gab$ isn't an
$\ab$-metric. For instance, as we know in section 1, the general
$\ab$-metric $F=\frac{(\sqrt{1+\bar b^2}\ba+\bb)^2}{\ba}$ is
actually an $\ab$-metric. So the following problem is still open:

Give an approach to distinguish $\ab$-metrics from general
$\ab$-metrics.\\
{\bf Acknowledgement} The authors thank Professor R. Bryant for his
immediate reply that there is indeed something left out in his paper
\cite{brya-some} when we ask him for suggestion. We also thank
Doctor Libing Huang. Before we introduce the concept of general
$\ab$-metrics, he has studied some Finsler metrics in the form
$F=|y|\p\left(|x|^2,\frac{\langle x,y\rangle}{|y|}\right)$ in a
different way, which are the simplest and most important general
$\ab$-metrics.

\noindent Changtao Yu\\
School of Mathematical Sciences, South China Normal
University\\
Guangzhou, 510631, P. R. China\\
aizhenli@gmail.com
\newline
\newline
\noindent Hongmei Zhu\\
School of Mathematical Sciences, Peking University\\
Beijing, 100871, P. R. China\\
zhuhongmei1981@pku.edu.cn

\end{document}